\documentclass[12pt, reqno]{amsart}
\usepackage[margin=1in]{geometry}
\usepackage{amssymb,latexsym,amsmath,amscd,amsfonts}
\usepackage{latexsym}
\usepackage[mathscr]{eucal}
\usepackage{bm}
\usepackage{mathptmx}
\usepackage{amssymb}
\usepackage{amsthm}
\usepackage{dcolumn}
\usepackage[all]{xy}
\usepackage{enumitem}
\usepackage[utf8]{inputenc}

\def \qed {\hfill \vrule height6pt width 6pt depth 0pt}
\def\textmatrix#1&#2\\#3&#4\\{\bigl({#1 \atop #3}\ {#2 \atop #4}\bigr)}
\def\dispmatrix#1&#2\\#3&#4\\{\left({#1 \atop #3}\ {#2 \atop #4}\right)}
\newcommand{\beg}{\begin{equation}}
	\newcommand{\eeg}{\end{equation}}
\newcommand{\ben}{\begin{eqnarray*}}
	\newcommand{\een}{\end{eqnarray*}}

\newtheorem{thm}{Theorem}[section]

\newtheorem{lem}[thm]{Lemma}

\newtheorem{prop}[thm]{Proposition}
\numberwithin{equation}{section} \theoremstyle{definition}
\newtheorem{defn}[thm]{Definition}
\newtheorem{rem}[thm]{Remark}

\newtheorem{eg}[thm]{Example}

\newcommand{\HS}{\mathcal H}
\newcommand{\KS}{\mathcal K}
\newcommand{\C}{\mathbb{C}}
\newcommand{\D}{\mathbb{D}}
\newcommand{\T}{\mathbb{T}}

\newcommand{\Q}{C_{1,r}}
\newcommand{\A}{\mathbb{A}_r}
\newcommand{\ov}{\overline}
\newcommand{\la}{\left\langle}
\newcommand{\ra}{\right\rangle}

\def\textmatrix#1&#2\\#3&#4\\{\bigl({#1 \atop #3}\ {#2 \atop #4}\bigr)}
\def\dispmatrix#1&#2\\#3&#4\\{\left({#1 \atop #3}\ {#2 \atop #4}\right)}

\begin{document}
	\title[On doubly commuting operators in $C_{1, r}$ class and quantum annulus]{On doubly commuting operators in $C_{1, r}$ class and quantum annulus}
	
	\author[Tomar] {NITIN TOMAR}
	
	\address[Nitin Tomar]{Mathematics Department, Indian Institute of Technology Bombay, Powai, Mumbai-400076, India.} \email{tomarnitin414@gmail.com}		
	
	
	
	\thanks{The author is supported by the Prime Minister's Research Fellowship (PMRF ID 1300140), Government of India.}	
	
	\begin{abstract}
		For $ 0 < r < 1 $, let $ \mathbb{A}_r = \{ z \in \mathbb{C} : r < |z| < 1 \} $ be the annulus with boundary $ \partial \overline{\mathbb{A}}_r = \mathbb{T} \cup r\mathbb{T} $, where $ \mathbb{T} $ is the unit circle in the complex plane $\mathbb C$. We study the class of operators 
		\[
		C_{1,r} = \{ T : T \text{ is invertible and } \|T\|, \|rT^{-1}\| \leq 1 \},
		\]
		introduced by Bello and Yakubovich. Any operator $T$ for which the closed annulus $\overline{\mathbb{A}}_r$ is a spectral set is in $C_{1,r}$.  The class $C_{1, r}$ is closely related to the \textit{quantum annulus} which is given by
		\[
		QA_r = \{ T : T \text{ is invertible and } \|rT\|, \|rT^{-1}\| \leq 1 \}.
		\]
		McCullough and Pascoe proved that an operator in $ QA_r $ admits a dilation to an operator $ S $ satisfying $(r^{-2} + r^2)I - S^*S - S^{-1}S^{-*} = 0$. An analogous dilation result holds for operators in $ C_{1,r}$ class. We extend these dilation results to doubly commuting tuples of operators in quantum annulus as well as in $C_{1,r}$ class. We also provide characterizations and decomposition results for such tuples.
	\end{abstract}

	\maketitle
	
	\section{Introduction}
	
	\vspace{0.2cm}
	
	\noindent In this paper, we find characterizations, dilations and decomposition results for an operator tuple consisting of operators from the $C_{1,r}$ class and quantum annulus $QA_r$, where
	\begin{align}
C_{1,r} &=\{T: T \ \mbox{is invertible and} \ \|T\|, \|rT^{-1}\| \leq 1\} \ \ 
\text{and} \label{eqn01} \\ 
QA_r&=\{T :  T \ \text{is invertible and} \ \|rT\|, \|rT^{-1}\| \leq 1 \} \label{eqn02}.	
			\end{align}
We shall follow standard terminologies: every operator is a bounded linear map on a complex Hilbert space and a \textit{contraction} is an operator with norm at most $1$. We denote by $\D , \T , r\T$ the unit disk, the unit circle and the circle with radius $r$ respectively with each of them having center at the origin in the complex plane $\C$. For a Hilbert space $\HS$, $\mathcal B(\HS)$ is the algebra of operators on $\HS$ and $I_\HS$ denotes the identity operator on a Hilbert space $\HS$. For $r \in (0,1)$, we consider the annuli
	\[
	\A=\{ z \in \C \,:\, r<|z|<1 \} \quad \text{and} \quad A_r=\{z \in \C \;:\, r<|z|<r^{-1} \}.
	\]
	In his seminal work, Sarason \cite{Sarason} developed a theory for $H^p$ spaces on an annulus, extending classical results for the disk and studying invariant subspaces for the multiplication operator on $L^2(\partial \A)$. Agler \cite{Agler} profoundly established the success of rational dilation on an annulus. Moreover, since the operators with $\ov{\mathbb A}_r$ as a spectral set belong to $C_{1,r}$ (see \eqref{eqn01}), the class $C_{1,r}$ naturally arises as a broader framework for further study. Beginning with the works of Shields \cite{Shields}, Misra \cite{Misra}, and Douglas and Paulsen \cite{Douglas}, the classes $C_{1,r}$ and $QA_r$ have been extensively studied in the literature with further contributions by Bello and Yakubovich \cite{Dmitry} and Mittal \cite{Mittal}. Later, McCullough, Pascoe and Tsikalas came up with deeper research results on the quantum annulus, e.g. see  \cite{Pas-McCull, TsikalasII}. However, in  \cite{N-S1} these two classes, i.e., the $C_{1,r}$ class and the quantum annulus, were shown to have a one-one correspondence between them in the following way.
	
	\begin{lem}[\cite{N-S1}, Lemma 1.5] \label{lem_basic}
		An operator $T \in C_{1, r}$ if and only if $r^{-1\slash 2}T \in QA_{\sqrt{r}}$. Also, $T \in QA_r$ if and only if $rT \in C_{1, r^2}$.
	\end{lem}
We say that a commuting tuple of invertible operators $(T_1, \dotsc, T_d)$ on a Hilbert space $\HS$ admits a \textit{dilation} to a tuple of  invertible operators $(J_1, \dotsc, J_d)$ acting on a Hilbert space $\KS$, if there is an isometry $V:\HS \to \KS$ such that 
\[
T_1^{n_1}\dotsc T_d^{n_d}=V^*J_1^{n_1}\dotsc J_d^{n_d}V
\] 
for all integers $n_1, \dotsc, n_d$. McCullough and Pascoe \cite{Pas-McCull} proved that an operator in quantum annulus admits a dilation to an operator $J$ on a space $\mathcal{K}$ that satisfies $(r^{-2}+r^2)I_\KS-J^*J-J^{-1}J^{-*}=0$. Using similar techniques as in \cite{Pas-McCull}, we generalize the results of \cite{Pas-McCull} to doubly commuting operators in $QA_r$ and $C_{1,r}$ classes. The following two theorems are the main results of this article. 

\begin{thm}\label{thm_mainII}
	Let $\underline{T}=(T_1, \dotsc, T_d)$ be a doubly commuting tuple of invertible operators acting on a Hilbert space $\HS$. Then the following are equivalent: 
	\begin{enumerate}
		\item $T_1, \dotsc, T_d \in \Q$;
		\item $\underline{T}$ admits a dilation to a doubly commuting tuple $\underline{J}=(J_1, \dotsc, J_d)$ of invertible operators acting on a space $\KS$ such that $(1+r^2)I_\KS-J_m^*J_m-r^2J_m^{-1}J_m^{-*}=0$ for $1\leq m\leq d$. 
	\end{enumerate}
\end{thm}

\begin{thm}\label{thm_mainQAr}
	Let $\underline{T}=(T_1, \dotsc, T_d)$ be a doubly commuting tuple of invertible operators acting on a Hilbert space $\HS$. Then the following are equivalent: 
	\begin{enumerate}
		\item $T_1, \dotsc, T_d \in QA_r$;
		\item $\underline{T}$ admits a dilation to a doubly commuting tuple $\underline{J}=(J_1, \dotsc, J_d)$ of invertible operators acting on a space $\KS$ such that $(r^{-2}+r^2)I_{\mathcal K}-J_m^*J_m-J_m^{-1}J_m^{-*}=0$ for $1 \leq m \leq d$.
	\end{enumerate}
\end{thm} 	
In Theorem \ref{thm_dc_decomp} \& Theorem \ref{thm_dc_decompII}, we find characterizations for a doubly commuting tuple of operators in $C_{1,r}$ class and $QA_r$ class respectively. In Theorem \ref{thm_dc_C1r}, we present the decomposition result for doubly commuting tuples of operators in $C_{1,r}$ class. One can then obtain an analogous decomposition result for doubly commuting operators in quantum annulus.
	\section{Dilation of doubly commuting operators in $\Q$ and quantum annulus}\label{sec_0004}
	
	\vspace{0.2cm}

	\noindent  In this Section, we present dilation results for finitely many doubly commuting operators in $C_{1,r}$ class and quantum annulus. To begin with, we recall from literature \cite{Pas-McCull, N-S1} that
	\begin{align*}
		QA_r&=\{T: \ \text{$T$ is invertible} \ (r^{-2}+r^2)I-T^*T-T^{-1}T^{-*} \geq 0 \} \ \ \text{and} \\
		C_{1,r}&=\{T: \ \text{$T$ is invertible} \ (1+r^2)I-T^*T-r^2T^{-1}T^{-*} \geq 0 \}.
	\end{align*}
Following the terminologies in \cite{Pas-McCull}, we define $\mathcal{Q}A_r$ to be the class of operators $S$ acting on a Hilbert space $\HS$ such that there exist orthogonal projections $P_0, P_1$ on $\HS$ satisfying $P_0+P_1=I_\HS$ and $S^*S=r^2P_0+r^{-2}P_1$. Evidently, $\mathcal{Q}A_r \subseteq QA_r$. The next result gives a better description of $\mathcal{Q}A_r$.
	
	\begin{prop}[\cite{Pas-McCull}, Proposition 3.2]\label{prop_3.2_in_Mc}
		An invertible operator $S$ acting on a Hilbert space $\HS$ is in $\mathcal{Q}A_r$ if and only if $(r^{-2}+r^2)I_\HS-S^*S-S^{-1}S^{-*}=0$.
	\end{prop}  	

One of the main results in \cite{Pas-McCull} is the following dilation result for $QA_r$ class.

\begin{thm}[\cite{Pas-McCull}, Theorem 1.1] \label{thm_main_an}
	Let $T$ be an invertible operator acting on a Hilbert space $\mathcal{H}$. Then
	$T \in QA_r$ if and only if  $T$ admits a dilation to an operator in $\mathcal{Q}A_r$ class.
\end{thm} 	
	
	Analogous to $\mathcal{Q}A_r$, we define $\mathcal{C}_{1,r}$ to be the class consisting of operators $J$ acting on a Hilbert space $\HS$ such that there exist orthogonal projections $P_0, P_1$ on $\HS$ satisfying $P_0+P_1=I_{\HS}$ and $J^*J=P_0+r^2P_1$. Consequently, we have an analog of Proposition \ref{prop_3.2_in_Mc} for $\mathcal{C}_{1, r}$ class. 
	
	\begin{prop}\label{prop_QSA}
		An invertible operator $J$ acting on a Hilbert space $\HS$ is in $\mathcal{C}_{1, r}$ if and only if  $(1+r^2)I_\HS-J^*J-r^2J^{-1}J^{-*}=0$. Moreover, $J \in \mathcal{C}_{1,r}$ if and only if $r^{-1\slash 2}J \in \mathcal{Q}A_{\sqrt{r}}$.
	\end{prop}
	
	The proof of Proposition \ref{prop_QSA} follows directly from the definition of $\mathcal{C}_{1, r}$ and Proposition \ref{prop_3.2_in_Mc}. Again a simple application of Lemma \ref{lem_basic} and Theorem \ref{thm_main_an} gives the following dilation result for $C_{1,r}$ class, which is an analog of Theorem \ref{thm_main_an}.
	
	\begin{thm}\label{thm_main}
		Let $T$ be an invertible operator acting on a Hilbert space $\mathcal{H}$. Then
		$T \in \Q$ if and only if $T$ admits a dilation to an operator in  $\mathcal{C}_{1, r}$ class.
	\end{thm}
	
 It is evident that Theorems \ref{thm_mainII} \&  \ref{thm_mainQAr} are generalizations of Theorems \ref{thm_main} \& \ref{thm_main_an} respectively which are credited to McCullough and Pascoe \cite{Pas-McCull}. We follow the similar techniques as in \cite{Pas-McCull} to obtain the desired dilation result. To do so, we begin with the following theorem which can be seen as a particular case of Theorem \ref{thm_mainII}. 
	
	\begin{thm}\label{lem_main}
		Let $(T_1, \dotsc, T_d)$ be a doubly commuting tuple of  operators in $\Q$ acting on a Hilbert space $\HS$. If $\sigma((T_j^*T_j)^{1\slash 2})$ is a finite subset of $(r, 1)$ for $ 1 \leq j \leq d$, then there exist a Hilbert space $\KS$, an isometry $V: \HS \to \KS$ and a doubly commuting tuple $(M_1, \dotsc, M_d)$ of operators in $\mathcal{C}_{1, r}$ acting on $\KS$ such that 
		$
		T_1^{n_1}\dotsc T_d^{n_d}=V^*M_1^{n_1}\dotsc M_d^{n_d}V
		$
		for all integers $n_1, \dotsc, n_d$. 
	\end{thm}

	\begin{proof}
		For $1 \leq j \leq d$, let us define $D_j=(T_j^*T_j)^{1\slash 2}$ and $U_j=T_jD_j^{-1}$. Each $D_j$ is an invertible positive operator since $T_j^*T_j$ is invertible. Note that 
		\[
		U_j^*U_j=D_j^{-1}T_j^*T_jD_j^{-1}=D_j^{-1}D_j^2D_j^{-1}=I_\HS \quad \text{and} \quad U_jU_j^*=T_jD_j^{-2}T_j^*=T_j(T_j^*T_j)^{-1}T_j^*=I_\HS.
		\]
		Thus, $U_j$ is a unitary and $T_j=U_jD_j$ for $1 \leq j \leq d$. Let $1 \leq i, j \leq d$ with $i \ne j$. The doubly commutativity hypothesis gives that $D_i^2D_j^2=D_j^2D_i^2, T_iD_j^2=D_j^2T_i$ and so by spectral theorem, we have that $D_iD_j=D_jD_i$ and $T_iD_j=D_jT_i$. Consequently, we have
		\begin{equation}\label{eqn_01}
		T_iD_j=D_jT_i, \quad 	D_iD_j=D_jD_i, \quad U_iU_j=U_jU_i \quad \text{and} \quad U_iD_j=D_jU_i \quad (i \ne j).
		\end{equation}
		Let $\underline{D}=(D_1, \dotsc, D_d)$ and let $\Omega=\sigma_T(\underline{D})$. We have by the projection property of Taylor-joint spectrum that $\Omega \subseteq \sigma(D_1) \times \dotsc \times \sigma(D_d)$ which is a finite subset of $(r, 1)^d$. It follows from the spectral theorem  that there are finitely many orthogonal projections $P_{j, 1}, \dotsc, P_{j, m_j}$ summing to the identity and $\lambda_{j, \alpha} \in (r, 1)$ for $1 \leq \alpha \leq m_j$ and $1 \leq j \leq d$ such that $D_j$ can be written as 
		\[
		D_j=\overset{m_j}{\underset{\alpha=1}{\sum}}\lambda_{j, \alpha} P_{j, \alpha} \qquad (1 \leq j \leq d).
		\]
		Indeed, a repeated application of the spectral theorem yields that
		\begin{equation}\label{eqn_02}
		 P_{i, \beta}P_{j, \alpha}=P_{j, \alpha}P_{i, \beta} \qquad \text{and} \qquad U_kP_{i, \beta}=P_{i, \beta}U_k
		\end{equation}
		for $1 \leq \alpha \leq m_j, 1 \leq \beta \leq m_i$ and $ 1 \leq i, j, k \leq d$ with $i \ne k$. Let $\T^+=\{e^{i\theta} : 0<\theta< \pi\}$ and  $\T^-=\{e^{i\theta} : -\pi< \theta < 0\}$ denote the upper and lower open arcs of $\T$ respectively. Following the proof of Lemma 3.1 in \cite{Pas-McCull}, one can choose an analytic surjective map $v: \D \to \A$ such that 
		\begin{equation}\label{eqn_03}
			v(0)=\sqrt{r},\quad  v(\T^+)=r\T , \quad v(\T^-)=\T
		\end{equation}
		and $v$ extends to $\T$ except at $\pm 1$. Given $\lambda_{j, \alpha} \in (r, 1)$, we have the existence of $w_{j, \alpha} \in \D$ such that 
		\begin{equation}\label{eqn_04}
			v(w_{j, \alpha})=\lambda_{j, \alpha} 
		\end{equation}
		for $1 \leq \alpha \leq m_j$ and $1 \leq j  \leq d$. Let us define 
		\[
		v_{j, \alpha}=v \circ \varphi_{j, \alpha}, \quad \text{where} \quad \varphi_{j, \alpha}(z)=\frac{w_{j, \alpha}-z}{1-\overline{w}_{j, \alpha}z} \quad (z \in \D)
		\]
		for $1 \leq \alpha \leq m_j$ and $1 \leq j  \leq d$. Note that $\varphi_{j, \alpha}$ is an automorphism of $\D$ with $\varphi_{j, \alpha}^{-1}=\varphi_{j, \alpha}$. Evidently, we have for each $\alpha$ and $j$ that
		\begin{equation}\label{eqn_05}
			v_{j, \alpha}(0)=v(\varphi_{j, \alpha}(0))=v(w_{j, \alpha})=\lambda_{j, \alpha}.
		\end{equation}
		Let $\T_{j, \alpha}^{+}=\varphi_{j, \alpha}(\T^+)$ and $\T_{j, \alpha}^{-}=\varphi_{j, \alpha}(\T^-)$. We have by (\ref{eqn_03}) that
		\begin{align}\label{eqn_06}
			v_{j, \alpha}(\T_{j, \alpha}^+)=v\circ \varphi_{j, \alpha}(\T_{j, \alpha}^+)=v(\T^+)=r\T \quad \text{and} \quad v_{j, \alpha}(\T_{j, \alpha}^-)=v\circ \varphi_{j, \alpha}(\T_{j, \alpha}^-)=v(\T^-)=\T.
		\end{align}
		For $1 \leq j \leq d$, we define
		\[
		\widehat{D}_j: \D \to \mathcal{B}(\HS), \quad \widehat{D}_j(z)=\overset{m_j}{\underset{\alpha=1}{\sum}}v_{j, \alpha}(z) P_{j, \alpha}.
		\]
		Each $\widehat{D}_j(z)$ is a normal operator since it is a sum of commuting normal operators. Also, $\widehat{D}_j(0)=D_j$ which follows from \eqref{eqn_02} and \eqref{eqn_05}. Moreover, $\widehat{D}_j(\zeta)$ exists for all but finitely many points $\zeta$ in $\T$. We have by (\ref{eqn_06}) that $v_{j, \alpha}(\zeta) \ne 0$ for all $\zeta \in \T$ except at finitely many points. Consequently, $\widehat{D}_j(\zeta)$ is invertible for all but finitely many points $\zeta \in \T$ and
		$\displaystyle 
		\widehat{D}_j(\zeta)^{-1}=\underset{1\leq \alpha \leq m_j}{\sum}v_{j, \alpha}(\zeta)^{-1} P_{j, \alpha},
		$
		whenever it exists. The invertibility of such $\widehat{D}_j(\zeta)$ follows from the fact that $\{P_{j, \alpha} : 1 \leq \alpha \leq m_j\}$ are mutually orthogonal projections that sum to the identity. It follows from (\ref{eqn_02}) that 
		\begin{equation}\label{eqn_07}
			\widehat{D}_j(z)U_k=U_k\widehat{D}_j(z) \quad \text{and} \quad \widehat{D}_j(z)\widehat{D}_i(z)=\widehat{D}_i(z)\widehat{D}_j(z)
		\end{equation}
		for $1 \leq i, j, k \leq d$ with $j \ne k$ and $z \in \D$. Consider the function given by 
		\[
		F_j: \D \to \mathcal{B}(\HS), \qquad F_j(z)=U_j\widehat{D}_j(z)=\overset{m_j}{\underset{\alpha=1}{\sum}}v_{j, \alpha}(z) U_jP_{j, \alpha} \quad (1 \leq j \leq d)
		\]
and then for all $z \in \D$ and $i \ne j$, we have that
		\begin{equation}\label{eqn_08}
			F_i(z)F_j(z)=F_j(z)F_i(z), \quad F_i(z)F_j(z)^*=F_j(z)^*F_i(z) \quad \text{and} \quad F_j(0)=T_j.
		\end{equation}
		The above equalities in \eqref{eqn_08} follow directly from \eqref{eqn_02}. Furthermore, $F_j$ extends to $\T$ except at finitely many points which follows from the construction of $v_{j, \alpha}$ as seen in (\ref{eqn_03}) and thereafter. Since each $v_{j, \alpha}$ has an extension to $\T$ except at finitely many points, we have that \eqref{eqn_08} holds for almost everywhere on $\T$ as well. It follows from the invertibility of $\widehat{D}_j(\zeta)$ almost everywhere on $\T$ that $F_j(\zeta)$ defines an invertible operator almost everywhere on $\T$. Let $Q_{j, \alpha}^{\pm}: L^2(\T) \to L^2(\T)$ be given by $Q_{j, \alpha}^{\pm}(f)=\chi_{j, \alpha}^{\pm}f$, where $\chi_{j, \alpha}^{\pm}$ are the characteristic functions of the arcs $\T_{j , \alpha}^{\pm}$ for $1 \leq j \leq d$ and $1 \leq \alpha \leq m_j$. Consider the space 
		$
		\mathcal{K}=L^2(\T) \otimes \HS.
		$ 
		Let $1 \leq j \leq d$. Since $\T_{j, \alpha}^+ \cup \T_{j, \alpha}^-$ is a subset of $\T$ and its complement in $\T$ is a finite set, we have $\chi_{j, \alpha}^+f + \chi_{j, \alpha}^-f=f$ for all $f \in L^2(\T)$ and $1 \leq \alpha \leq m_j$. Consequently, it follows that 
		\begin{small}
			\begin{equation*}
				\begin{split}
					\overset{m_j}{\underset{\alpha=1}{\sum}} Q_{j, \alpha}^{\pm} \otimes P_{j, \alpha}= \overset{m_j}{\underset{\alpha=1}{\sum}} (Q_{j, \alpha}^+ \otimes P_{j, \alpha})+(Q_{j, \alpha}^- \otimes P_{j, \alpha})=
					\overset{m_j}{\underset{\alpha=1}{\sum}} (Q_{j, \alpha}^++Q_{j, \alpha}^-) \otimes P_{j, \alpha}=  I_{L^2(\T)} \otimes \overset{m_j}{\underset{\alpha=1}{\sum}}P_{j, \alpha}=I_\mathcal{K},
				\end{split}
			\end{equation*}
		\end{small}
		where we have used the property that $\displaystyle \overset{m_j}{\underset{\alpha=1}{\sum}} P_{j, \alpha}=I_\HS$. Hence, the set $ \left\{ Q_{j, \alpha}^{\pm} \otimes P_{j, \alpha} : 1 \leq \alpha \leq m_j \right\} $ consists of pairwise orthogonal projections and  
		\begin{equation}\label{eqn_09}
			\overset{m_j}{\underset{\alpha=1}{\sum}} Q_{j, \alpha}^{\pm} \otimes P_{j, \alpha}=I_\mathcal{K}.
		\end{equation}
		For $1 \leq j \leq d$, let $M_j$ be the operator of multiplication by $F_j$ on $\mathcal{K}$. Fix $j \in \{1, \dotsc, d\}$. Note that $M_j$ is invertible and  $M_j^{-1}(f)(\zeta)=F_j(\zeta)^{-1}f(\zeta)$
		for all but finitely many points $\zeta \in \T$. The doubly commutativity of $M_1, \dotsc, M_d$ is a direct consequence of (\ref{eqn_08}) and the fact that $F_1(\zeta), \dotsc, F_d(\zeta)$ exist for all but finitely many points $\zeta \in \T$. Let $f \in L^2(\T), h \in \HS$ and $1 \leq \alpha \leq m_j$. For $\zeta \in \T$ such that $v_{j, \alpha}(\zeta)$ exists (which is a complement of a finite set), we have that 
		\begin{align}\label{eqn_10}
				M_j(Q_{j, \alpha}^{\pm} \otimes P_{j, \alpha})(f \otimes h) \zeta
				&=F_j(\zeta)(\chi_{j, \alpha}^{\pm}\left(\zeta)f(\zeta) \otimes P_{j, \alpha}h\right) \notag \\
				&=\left(\overset{m_j}{\underset{\beta=1}{\sum}}v_{j, \beta}(\zeta) U_jP_{j, \beta}\right)\left(\chi_{j, \alpha}^{\pm}(\zeta)f(\zeta) \otimes P_{j, \alpha}h\right) \notag \\
				&=\overset{m_j}{\underset{\beta=1}{\sum}}\left[\left(v_{j, \beta}(\zeta)\chi_{j, \alpha}^{\pm}(\zeta)f(\zeta)\right) \otimes \left( U_jP_{j, \beta} P_{j, \alpha}h\right)\right] \notag \\
				&=v_{j, \alpha}(\zeta)\chi_{j, \alpha}^{\pm}(\zeta)f(\zeta)\otimes U_jP_{j, \alpha}h,
		\end{align}
		where the last equality holds as $\{P_{j, \beta}: 1 \leq \beta \leq m_j\}$ consists of pairwise orthogonal projections. Hence, if $g \in L^2(\T)$ and $h' \in \HS$, then
		\begin{align}\label{eqn_11}
				&  \quad \la M_j^*M_j(Q_{j, \alpha}^{\pm} \otimes P_{j, \alpha})(f \otimes h)(\zeta), (g \otimes h')(\zeta) \ra \notag\\
				& =\la M_j(Q_{j, \alpha}^{\pm} \otimes P_{j, \alpha})(f \otimes h)(\zeta), M_j(g \otimes h')(\zeta) \ra \notag \\
				& =\la M_j(Q_{j, \alpha}^{\pm} \otimes P_{j, \alpha})(f \otimes h)(\zeta), \  M_j\left(\overset{m_j}{\underset{\beta=1}{\sum}} Q_{j, \beta}^{\pm} \otimes P_{j, \beta}\right)(g \otimes h')(\zeta) \ra \quad [\text{by} \ (\ref{eqn_09})] \notag \\
				& =\overset{m_j}{\underset{\beta=1}{\sum}} \la v_{j, \alpha}(\zeta)\chi_{j, \alpha}^{\pm}(\zeta)f(\zeta)\otimes U_jP_{j, \alpha}h, \  v_{j, \beta}(\zeta)\chi_{j, \beta}^{\pm}(\zeta)g(\zeta)\otimes U_jP_{j, \beta}h' \ra \quad [\text{by} \ (\ref{eqn_10})] \notag \\
				& =\overset{m_j}{\underset{\beta=1}{\sum}} \la v_{j, \alpha}(\zeta)\chi_{j, \alpha}^{\pm}(\zeta)f(\zeta), v_{j, \beta}(\zeta)\chi_{j, \beta}^{\pm}(\zeta)g(\zeta) \ra \la  U_jP_{j, \alpha}h, U_jP_{j, \beta}h' \ra \notag \\
				& =\overset{m_j}{\underset{\beta=1}{\sum}} \la v_{j, \alpha}(\zeta)\chi_{j, \alpha}^{\pm}(\zeta)f(\zeta), v_{j, \beta}(\zeta)\chi_{j, \beta}^{\pm}(\zeta)g(\zeta) \ra \la  P_{j, \alpha}h, P_{j, \beta}h' \ra  \quad [\text{as}  \ U_j \ \text{is unitary}] \notag \\
				& = |v_{j, \alpha}(\zeta)|^2\la \chi_{j, \alpha}^{\pm}(\zeta)f(\zeta), \chi_{j, \alpha}^{\pm}(\zeta)g(\zeta) \ra \la  P_{j, \alpha}h, P_{j, \alpha}h' \ra  \quad \left[\text{since} \ \{P_{j, \beta}\}_{\beta=1}^{m_j} \ \text{are orthogonal}\right] \notag \\
				&=|v_{j, \alpha}(\zeta)|^2\ \chi_{j, \alpha}^{\pm}(\zeta)f(\zeta)\overline{g(\zeta)} \la P_{j, \alpha}h, h' \ra .
			\end{align}
		Consequently, we have by (\ref{eqn_11}) that 
		\begin{align}\label{eqn_12} 
				\la M_j^*M_j(Q_{j, \alpha}^{+} \otimes P_{j, \alpha})(f \otimes h), (g \otimes h') \ra_{\mathcal{K}}
				&=\frac{1}{2\pi}\underset{\T}{\int} \la M_j^*M_j(Q_{j, \alpha}^{+} \otimes P_{j, \alpha})(f \otimes h)(\zeta), (g \otimes h')(\zeta) \ra d\zeta  \notag \\
				&=\frac{1}{2\pi}\underset{\T}{\int} |v_{j, \alpha}(\zeta)|^2\ \chi_{j, \alpha}^{+}(\zeta)f(\zeta)\overline{g(\zeta)} \la P_{j, \alpha}h, h' \ra  d\zeta   \notag \\
				&=\frac{r^2}{2\pi}\underset{\T_{j, \alpha}^+}{\int}f(\zeta)\overline{g(\zeta)} \la P_{j, \alpha}h, h' \ra  d\zeta  \qquad [\text{by} \ (\ref{eqn_06})]  \notag \\
				&=r^2\left[\frac{1}{2\pi}\underset{\T_{j, \alpha}^+}{\int}f(\zeta)\overline{g(\zeta)} \la P_{j, \alpha}h, h' \ra  d\zeta\right]  \notag \\
				&=r^2\left[\frac{1}{2\pi}\underset{\T}{\int}\la \chi_{j, \alpha}^+(\zeta)f(\zeta), g(\zeta)\ra \la P_{j, \alpha}h, h' \ra  d\zeta\right]  \notag \\
				&=r^2\left[\frac{1}{2\pi}\underset{\T}{\int}\la \chi_{j, \alpha}^+(\zeta)f(\zeta) \otimes P_{j, \alpha}h, \  g(\zeta)\otimes h' \ra d\zeta\right]  \notag \\
				&=r^2\la (Q_{j, \alpha}^{+} \otimes P_{j, \alpha})(f \otimes h), (g \otimes h')  \ra_{\mathcal{K}}.
			\end{align}
		Following the similar arguments as above, one can easily show that
		\begin{equation}\label{eqn_13}
			\begin{split}
				\la M_j^*M_j(Q_{j, \alpha}^{-} \otimes P_{j, \alpha})(f \otimes h), (g \otimes h') \ra_{\mathcal{K}}=\la (Q_{j, \alpha}^{-} \otimes P_{j, \alpha})(f \otimes h), (g \otimes h')  \ra_{\mathcal{K}}. \\
			\end{split} 
		\end{equation}
		For $1 \leq j \leq d$, let us define 
		\[
		Q_j^+=\overset{m_j}{\underset{\alpha=1}{\sum}} Q_{j, \alpha}^+ \otimes P_{j, \alpha} \qquad \text{and} \qquad Q_j^-=\overset{m_j}{\underset{\alpha=1}{\sum}} Q_{j, \alpha}^- \otimes P_{j, \alpha}.
		\]
	Let $f \in L^2(\T), \zeta \in \T$ and let $h \in \HS$. Then
	\begin{equation*}
		\begin{split}
Q_j^+ Q_j^-(f \otimes h)(\zeta)
&=\overset{m_j}{\underset{\alpha=1}{\sum}}  Q_j^+ (Q_{j, \alpha}^- \otimes P_{j, \alpha})(f \otimes h)(\zeta)	\\
&=\overset{m_j}{\underset{\alpha=1}{\sum}}  Q_j^+ (\chi_{j, \alpha}^-(\zeta)f(\zeta) \otimes P_{j, \alpha}h)	\\
&=\overset{m_j}{\underset{\alpha=1}{\sum}} \ \overset{m_j}{\underset{\beta=1}{\sum}} (Q_{j, \beta}^+ \otimes P_{j, \beta}) (\chi_{j, \alpha}^-(\zeta)f(\zeta) \otimes P_{j, \alpha}h)\\
&=\overset{m_j}{\underset{\alpha=1}{\sum}} \overset{m_j}{\underset{\beta=1}{\sum}} \chi_{j, \beta}^+(\zeta) \chi_{j, \alpha}^-(\zeta)f(\zeta) \otimes P_{j, \beta}P_{j, \alpha}h\\
&=\overset{m_j}{\underset{\alpha=1}{\sum}}\chi_{j, \alpha}^+(\zeta) \chi_{j, \alpha}^-(\zeta)f(\zeta) \otimes P_{j, \alpha}h \quad \left[ \text{as} \ \{P_{j, \beta}\}_{\beta=1}^{m_j} \ \text{are pairwise orthogonal}\right]\\
&=0,
		\end{split}
	\end{equation*}
where the last inequality follows from the fact that $\chi_{j, \alpha}^+(\zeta) \chi_{j, \alpha}^-(\zeta)=0$ for all $\zeta \in \T$ except at finitely many points. Similarly, one can show that $Q_j^+, Q_j^-$ are projections. We have by (\ref{eqn_09}) that $Q_j^+ + Q_j^-=I_{\mathcal{K}}$. Thus, $Q_j^+ \mathcal{K}$ and $Q_j^- \mathcal{K}$ are orthogonal subspaces of $\mathcal{K}$ and $\mathcal{K}=\overline{Q_j^+\mathcal{K}} \oplus \overline{Q_j^- \mathcal{K}}$. We have by (\ref{eqn_12}) and (\ref{eqn_13}) that $M_j^*M_jQ_j^+=r^2Q_j^+$ and $M_j^*M_jQ_j^-=Q_j^-$ respectively. We can write
		\[
		M_j^*M_j=\begin{bmatrix}
			r^2I & 0 \\
			0 & I
		\end{bmatrix}
		\]
		with respect to  $\mathcal{K}=\overline{Q_j^+\mathcal{K}} \oplus \overline{Q_j^- \mathcal{K}}$ and so, $M_j^*M_j+r^2(M_j^*M_j)^{-1}=(1+r^2)I_{\mathcal{K}}$. It follows from Proposition \ref{prop_QSA} that $M_j \in \mathcal{C}_{1, r}$ for $1 \leq j \leq d$.  Define $V: \HS \to \mathcal{K}$ by $Vh=1\otimes h$ and so, $V$ is an isometry.  Let $h, h' \in \HS, \zeta \in \T$ and let $n=(n_1, \dotsc, n_d) \in \mathbb{Z}^d$.  Then
		\begin{align}\label{eqn_14}
				\la M_1^{n_1}\dotsc M_d^{n_d}Vh(\zeta), Vh'(\zeta) \ra 
				&= \la M_1^{n_1}\dotsc M_d^{n_d}(1\otimes h)(\zeta), (1\otimes h')(\zeta) \ra \notag \\
				&= \la F_1^{n_1}(\zeta)\dotsc F_d^{n_d}(\zeta)(1 \otimes h), 1\otimes h' \ra \notag \\
				&= \la 1 \otimes F_1^{n_1}(\zeta)\dotsc F_d^{n_d}(\zeta)h, 1 \otimes h' \ra \notag \\
				&= \la F_1^{n_1}(\zeta)\dotsc F_d^{n_d}(\zeta)h, h' \ra \notag \\
				&= \lim_{s \to 1^-} \la F_1^{n_1}(s\zeta)\dotsc F_d^{n_d}(s\zeta)h,h' \ra.
			\end{align}
		Note that the map
		$
		z \mapsto \la F_1^{n_1}(z)\dotsc F_d^{n_d}(z)h,h' \ra
		$
		is a holomorphic map on $\D$. An application of dominated convergence theorem gives the following:
		\begin{equation*}
			\begin{split}
				\la M_1^{n_1}\dotsc M_d^{n_d}Vh, Vh' \ra
				&= \frac{1}{2\pi} \underset{\T}{\int}\la M_1^{n_1}\dotsc M_d^{n_d}Vh(\zeta), Vh'(\zeta) \ra  d\zeta \\
				&= \frac{1}{2\pi} \underset{\T}{\int} \lim_{s \to 1^-} \la  F_1^{n_1}(s\zeta)\dotsc F_d^{n_d}(s\zeta)h,h' \ra d\zeta \quad [\text{by} \ (\ref{eqn_14})] \\
				&= \lim_{s \to 1^-} \left[\frac{1}{2\pi} \underset{\T}{\int}  \la  F_1^{n_1}(s\zeta)\dotsc F_d^{n_d}(s\zeta)h,h' \ra d\zeta \right] \\
				&=\la  F_1^{n_1}(0)\dotsc F_d^{n_d}(0)h,h' \ra \qquad [\text{by Cauchy's integral formula}] \\
				&=\la  T_1^{n_1}\dotsc T_d^{n_d}h,h' \ra,
			\end{split}
		\end{equation*}
		where the last equality follows from (\ref{eqn_08}). Hence, $V^* M_1^{n_1}\dotsc M_d^{n_d}V=T_1^{n_1}\dotsc T_d^{n_d}$ for all integers $n_1, \dotsc, n_d$ which completes the proof.
	\end{proof}
	
To this end, we mention that following the proof of Theorem \ref{thm_main}, each $M_j$ in the proof of Theorem \ref{thm_main} becomes a normal operator if $(T_1, \dotsc, T_d)$ consists of normal operators. Now, we present the main result of this Section.

\medskip

	\noindent \textbf{Proof of Theorem \ref{thm_mainII}.} 
	$(1) \implies (2)$. Let $\underline{T}=(T_1, \dotsc, T_d)$ be a doubly commuting tuple of operators in $C_{1,r}$.  Consider the polar decomposition $T_j=U_j(T_j^*T_j)^{1\slash 2}$ for $j=1, \dotsc, d$. Let $D_j=(T_j^*T_j)^{1\slash 2}$ and let $\underline{D}=(D_1, \dotsc, D_d)$. The invertibility of each $D_j$ follows from the invertibility of $T_j$. Proceeding as in Theorem \ref{thm_main}, we have from the doubly commutativity hypothesis that $D_i^2D_j^2=D_j^2D_i^2$ and $T_iD_j^2=D_j^2T_i$. By spectral theorem, we have that $D_iD_j=D_jD_i$ and $T_iD_j=D_jT_i$. Consequently, we have
	\begin{equation*}
		T_iD_j=D_jT_i, \quad 	D_iD_j=D_jD_i, \quad U_iU_j=U_jU_i \quad \text{and} \quad U_iD_j=D_jU_i \quad (i \ne j).
	\end{equation*} 
Again, it follows from spectral theorem that there exists a unique spectral measure $E$ supported at $\Omega=\sigma_T(\underline{D}) \subseteq [r, 1]^d$ such that
	\[
	D_j=\underset{\Omega}{\int}\lambda_j \ dE(\lambda)=\underset{\sigma(D_j)}{\int}z \ dE_j(z) \qquad (j=1, \dotsc, d),
	\]
	where $E_j$ is the associated unique spectral measure relative to $(\sigma(D_j), \HS)$ (which can be recaptured from the spectral measure $E$) and $\lambda_j: \sigma_T(\underline{D}) \to \sigma(D_j)$ is the projection map onto the $j$-th coordinate. Choose $m \in \mathbb{N}$ and pick a simple measurable function $s_{m, j}$ taking values in $(r, 1)$ that approximates the map $z \mapsto z$ uniformly within $1\slash 2^m$ on $[r, 1]$, i.e., 
	\begin{equation}\label{eqn_15}
		\|s_{m, j}-z\|_{\infty, [r,1]}=\sup_{z \in [r,1]}|s_{m, j}(z)-z| \leq \frac{1}{2^m} \qquad (1 \leq j \leq d).
	\end{equation}
For $m \in \mathbb{N}$ and  $1 \leq j \leq d$, we define
	$
	D_{m, j}=s_{m,j}(D_j)$ and $T_{m, j}=U_jD_{m, j}$.
	Since each $s_{m, j}$ is a simple function with values in $(r,1)$, the operator $D_{m, j}$ can be written as a linear sum of finitely many commuting orthogonal projections whose sum is identity and $\sigma(D_{m, j})$ is a finite subset of $(r,1)$. 
	Note that $\|T_{m,j}\|=\|D_{m, j}\|$ and $\|T_{m, j}^{-1}\|=\|D_{m, j}^{-1}\|$ because each $U_j$ is a unitary. An application of the spectral theorem (as is done to obtain \eqref{eqn_02}), one can easily prove that 
	\[
	U_kD_{m, j}=D_{m, j}U_k, \quad D_iD_{m, j}=D_{m, j}D_i, \quad D_{m, i}D_{m, j}=D_{m, j}D_{m, i} \quad \text{and} \quad T_{m, i}T_{m, j}=T_{m, j}T_{m, i}
	\]
	for all $1 \leq i, j, k \leq d$ with $k \ne j$. We have by spectral theorem that $\|D_{m, j}-D_j\| \leq \|s_{m,j}-z\|_{\infty, [r,1]}$ for $m \in \mathbb{N}$ and $1\leq j\leq d$. We now have by (\ref{eqn_15}) that
	\begin{equation}\label{eqn_16}
		\|T_{m, j}-T_j\|= \|U_j(D_{m, j}-D_j)\| \leq \|D_{m, j}-D_j\| \leq \|s_{m,j}-z\|_{\infty, [r,1]} \leq \frac{1}{2^m} \to 0 \quad \text{as} \ m \to \infty
	\end{equation}
	and similarly, we have
	\begin{align}\label{eqn_17}
			\|T_{m,j}^{-1}-T_j^{-1}\| \leq \|D_{m,j}^{-1}-D_j^{-1}\|
			&=\|D_{m,j}^{-1}D_j^{-1}(D_j-D_{m,j})\| \notag \\
			&\leq \|D_{m,j}^{-1}\|\ \|D_j^{-1}\|\|D_j-D_{m,j}\| \notag \\
			& \leq \frac{r^{-2}}{2^m} \to 0 \quad \text{as} \ \ m \to \infty.
		\end{align}
	The last inequality follows because $D_j, D_{m,j}$ are self-adjoint operators with spectrum inside $[r,1]$ and so, $\|D_j^{-1}\|, \|D_{m,j}^{-1}\| \leq 1\slash r$. Putting everything together, we have that $(T_{m, 1}, \dotsc, T_{m, d})$ is a doubly commuting tuple of operators in $\Q$ and $\sigma((T_{m, j}^*T_{m, j})^{1\slash 2})$ is a finite subset of $(r, 1)$ for every $m \in \mathbb{N}$ and $1 \leq j \leq d$. It follows from Theorem \ref{lem_main} that there exist a Hilbert space $\KS_m$, an isometry $V_m: \HS \to \KS_m$ and a doubly commuting tuple $(M_{m, 1}, \dotsc, M_{m, d})$ of operators in $\mathcal{C}_{1, r}$ acting on $\KS_m$ such that for all integers $n_1, \dotsc, n_d$, we have
	\begin{equation}\label{eqn_18}
		T_{m,1}^{n_1}\dotsc T_{m,d}^{n_d}=V_m^*M_{m,1}^{n_1}\dotsc M_{m,d}^{n_d}V_m.
	\end{equation}
	For $1 \leq i \leq d$, we define $M_i=\underset{m\in \mathbb{N}}{\bigoplus}M_{m, i}$ acting on the Hilbert space $\mathcal{K}$ given by
	\[
	\mathcal{K}=\underset{m\in \mathbb{N}}{\bigoplus}\mathcal{K}_{m}=\left\{(k_1, k_2, \dotsc)\ : \ k_m \in \mathcal{K}_m \ \text{and}  \ \overset{\infty}{\underset{m=1}{\sum}}\|k_m\|^2 <\infty  \right\}.
	\]
	In other words, $M_i(k_1, k_2, \dotsc, )=(M_{1, i}(k_1), M_{2, i}(k_2), \dotsc )$ for $(k_1, k_2, \dotsc) \in \mathcal{K}$ and $1 \leq i \leq d$. By definition, it follows that $\underline{M}=(M_1, \dotsc, M_d)$ is a doubly commuting tuple of operators in $\mathcal{C}_{1, r}$.
	
	\medskip 
	
	Let $\mathcal{A}$ be the closure of the self-adjoint subspace $\mathcal{A}_0$ of $\mathcal{B}(\KS)$, which is given by
	\[
	\mathcal{A}_0=\text{span}\left\{p(\underline{M}), q(\underline{M})^* : p, q \ \text{are Laurent polynomials in $d$-variables} \right\}.
	\]
By a Laurent polynomial in $d$-variables, we mean a polynomial in the variables $z_1, z_1^{-1}, \dotsc, z_d, z_d^{-1}$. Evidently, $\mathcal{A}_0$ is an operator system in $\mathcal{B}(\KS)$.	Define $\Phi:\mathcal{A}_0 \to \mathcal{B}(\HS)$ as $\Phi(p(\underline{M}))=p(\underline{T})$ and $\Phi(p(\underline{M})^*)=p(\underline{T})^*$ which is extended linearly to $\mathcal{A}_0$. We claim that $\Phi$ is a completely positive map for which it suffices to show that 
	\[
	\begin{bmatrix}
		p_{\alpha\beta}(\underline{M})+q_{\alpha\beta}(\underline{M})^*
	\end{bmatrix}_{\alpha, \beta=1}^N \geq 0 \implies
	\begin{bmatrix}
		p_{\alpha\beta}(\underline{T})+q_{\alpha\beta}(\underline{T})^*
	\end{bmatrix}_{\alpha, \beta=1}^N \geq 0
	\]
	for all $N \in \mathbb{N}$ and Laurent polynomials $p_{\alpha \beta}, q_{\alpha \beta}$. Here, 
	\[
	\begin{bmatrix}
		p_{\alpha\beta}(\underline{M})+q_{\alpha\beta}(\underline{M})^*
	\end{bmatrix}_{\alpha, \beta=1}^N, \begin{bmatrix}
	p_{\alpha\beta}(\underline{T})+q_{\alpha\beta}(\underline{T})^*
	\end{bmatrix}_{\alpha, \beta=1}^N \geq 0
	\] 
	indicate that the corresponding block matrices are positive semi-definite as operators on $\KS^N$ and $\HS^N$, respectively. It follows from (\ref{eqn_16}) and (\ref{eqn_17}) that 
	\[
	\lim_{m\to \infty}\|p(T_{m,1}, \dotsc, T_{m,d})- p(T_1, \dotsc, T_d)\| = 0
	\]
	for all Laurent polynomials $p$ in $d$-variables. Take $N \in \mathbb{N}$. Let $p_{\alpha \beta}$ and $q_{\alpha \beta}$ be Laurent polynomials in $d$-variables for all $1 \leq \alpha, \beta \leq N$. Consequently, we have that
	\begin{equation*}
		\begin{split}
			& \quad \quad 0 \leq  \begin{bmatrix}
				p_{\alpha\beta}(\underline{M})+q_{\alpha\beta}(\underline{M})^*
			\end{bmatrix}_{\alpha, \beta=1}^N
			= \bigoplus_{m \in \mathbb N}\begin{bmatrix}
				p_{\alpha\beta}(M_{m,1}, \dotsc, M_{m,d})+q_{\alpha\beta}(M_{m,1}, \dotsc, M_{m,d})^*
			\end{bmatrix}_{\alpha, \beta=1}^N\\
			& \implies 0 \leq \begin{bmatrix}
				p_{\alpha\beta}(M_{m,1}, \dotsc, M_{m,d})+q_{\alpha\beta}(M_{m,1}, \dotsc, M_{m,d})^*
			\end{bmatrix}_{\alpha, \beta=1}^N \quad \text{for all $m \in \mathbb{N}$}\\
			&\implies 0 \leq \begin{bmatrix}
				V_m^*p_{\alpha\beta}(M_{m,1}, \dotsc, M_{m,d})V_m+V_m^*q_{\alpha\beta}(M_{m,1}, \dotsc, M_{m,d})^*V_m
			\end{bmatrix}_{\alpha, \beta=1}^N \quad \text{for all $m \in \mathbb{N}$}\\
			&\implies 0 \leq \begin{bmatrix}
				p_{\alpha\beta}(T_{m,1}, \dotsc, T_{m,d})+q_{\alpha\beta}(T_{m,1}, \dotsc, T_{m,d})^*
			\end{bmatrix}_{\alpha, \beta=1}^N \quad \text{for all $m \in \mathbb{N}$} \quad [\text{by} \ (\ref{eqn_18})]\\
		\end{split}
	\end{equation*}
	and taking limit $m \to \infty$, we get that $  \begin{bmatrix}
		p_{\alpha\beta}(\underline{T})+q_{\alpha\beta}(\underline{T})^*
	\end{bmatrix}_{\alpha, \beta=1}^N \geq 0$. Thus $\Phi$ is a unital completely positive linear map. One can extend the map $\Phi$ continuously on $\mathcal{A}$ which is the closure of the operator system $\mathcal{A}_0$. By Arveson's extension theorem (see \cite{Arveson}), there is a completely positive linear map $\Psi: \mathcal{B}(\KS) \to \mathcal{B}(\HS)$ such that $\Phi=\Psi|_\mathcal{A}$. Clearly, $\Psi$ is unital. It follows from Stinespring's dilation theorem (see \cite{Stine}) that there exist a Hilbert space $\mathcal{L}$, an isometry $V: \mathcal{H} \to \mathcal{L}$ and a unital $*$-representation $\pi:\mathcal{B}(\KS) \to \mathcal{B}(\mathcal{L})$ such that  
	$V^*\pi(S)V=\Psi(S)$ for all $S \in \mathcal{B}(\KS)$. In particular, $V^*\pi(S)V=\Phi(S)$ for all $S \in \mathcal{A}$. Hence, for any Laurent polynomial $p$, we have that 
	\[
	p(\underline{T})=\Phi(p(\underline{M}))=V^*\pi(p(\underline{M}))V.
	\]
	Define $(J_1, \dotsc, J_d)=(\pi(M_1), \dotsc, \pi(M_d)) \in \mathcal{B}(\mathcal{L})^d$. Using the $*$-homomorphism property of $\pi$ and the fact that $(M_1, \dotsc, M_d)$ is a doubly commuting tuple of operators in $\mathcal{C}_{1, r}$, we have that $(J_1, \dotsc, J_d)$ is a doubly commuting tuple of operators in $\mathcal{C}_{1, r}$. Moreover, $T_1^{n_1} \dotsc T_d^{n_d}=V^* J_1^{n_1} \dotsc J_d^{n_d} V$ for all integers $n_1, \dotsc, n_d$. \smallskip 
	
\medskip 

\noindent $(2) \implies (1)$. Suppose $\underline{T}$ admits a dilation to a doubly commuting tuple $\underline{J}=(J_1, \dotsc, J_d)$ of invertible operators acting on a Hilbert space $\KS$ such that  $(1+r^2)I_{\mathcal K}-J_m^*J_m-r^2J_m^{-1}J_m^{-*}=0$ for $1 \leq m \leq d$. Let $m \in \{1, \dotsc, d\}$. Proposition \ref{prop_QSA} guarantees the existence of orthogonal projections $P_0, P_1$ on $\KS$ that sum to $I_\KS$ and $J_m^*J_m=P_0+r^{2}P_1$. Then $\|T_m\|^2 \leq \|J_m\|^2=\|J_m^*J_m\| \leq 1$ and $\|T_m^{-1}\|^2 \leq \|J_m^{-1}\|^2=\|J_m^{-1}J_m^{-*}\| \leq r^{-2}$. Hence, $T_m \in C_{1, r}$. The proof is now complete. \qed 
	
\vspace{0.2cm}

As an application of Theorem \ref{thm_mainII}, we now present the proof of Theorem \ref{thm_mainQAr}.

\vspace{0.2cm}

\noindent \textbf{Proof of Theorem \ref{thm_mainQAr}.}
$(1) \implies (2)$. Let $T_i \in QA_r$ for $1 \leq i \leq d$. We have by Lemma \ref{lem_basic} that $A_i=rT_i \in C_{1, r^2}$ and so, $\underline{A}=(A_1, \dotsc, A_d)$ is a doubly commuting tuple of operators in $C_{1, r^2}$ class. By Theorem \ref{thm_mainII}, there exist a Hilbert space $\KS$, an isometry $V: \HS \to \KS$ and a doubly commuting tuple $(S_1, \dotsc, S_d)$ of operators in $\mathcal{C}_{1, r^2}$ acting on $\KS$ such that 
\[
A_1^{n_1}\dotsc A_d^{n_d}=V^*S_1^{n_1}\dotsc S_d^{n_d}V \quad \text{and so,} \quad T_1^{n_1}\dotsc T_d^{n_d}=V^*(r^{-1}S_1)^{n_1}\dotsc (r^{-1}S_d)^{n_d}V
\]
for all integers $n_1, \dotsc, n_d$. Define the operator tuple $\underline{J}=(J_1, \dotsc, J_d)=(r^{-1}S_1, \dotsc, r^{-1}S_d)$. It is evident that $\underline{J}$ dilates $\underline{T}$. For $m \in \{1, \dotsc, d\}$, we have that
\begin{equation*}
	\begin{split}
	(r^{-2}+r^2)I_{\mathcal K}-J_m^*J_m-J_m^{-1}J_m^{-*}
	=	(r^{-2}+r^2)I_{\mathcal K}-r^{-2}S_m^*S_m-r^2S_m^{-1}S_m^{-*}
	%=	r^{-2}\left((1+r^4)I_{\mathcal K}-S_m^*S_m-r^4S_m^{-1}S_m^{-*}\right)
	=0,
	\end{split}
\end{equation*}
where the last equality follows from Proposition \ref{prop_QSA} and the fact that $S_m \in \mathcal{C}_{1, r^2}$ for $1 \leq m \leq d$. Consequently, we have the desired conclusion as in part $(2)$.

\medskip 

\noindent $(2) \implies (1)$. Suppose there exists a tuple $\underline{J}=(J_1, \dotsc, J_d)$ on a Hilbert space $\KS$ such that $\underline{J}$ dilates $\underline{T}$ and $(r^{-2}+r^2)I_{\mathcal K}-J_m^*J_m-J_m^{-1}J_m^{-*}=0$ for $1 \leq m \leq d$. Let $m \in \{1, \dotsc, d\}$. By Proposition \ref{prop_3.2_in_Mc}, there exist orthogonal projections $P_0, P_1$ that sum to the identity and $J_m^*J_m=r^2P_0+r^{-2}P_1$. Then
\[
\|T_m\|^2\leq \|J_m\|^2=\|J_m^*J_m\| \leq r^{-2}, \quad \|T_m^{-1}\|^2\leq\|J_m^{-1}\|^2=\|J_m^{-1}J_m^{-*}\| \leq r^{-2} 
\]
and so, $T_m \in QA_r$. The proof is now complete. \qed 

\vspace{0.2cm}

It is well-known that a commuting tuple of normal operators is doubly commuting. Consequently, the dilation results for doubly commuting tuples in the $C_{1,r}$ class and quantum annulus, namely Theorems \ref{thm_mainII} and \ref{thm_mainQAr}, also apply to commuting tuples of normal operators in these classes.
	
	\section{Decomposition of doubly commuting operators in $\Q$ and quantum annulus} \label{sec_decomp}

\vspace{0.2cm}

\noindent Recall that a contraction $T$ on a Hilbert space $\HS$ is said to be \textit{completely non-unitary} if there is no non-zero closed subspace of $\HS$ which reduces $T$ to a unitary. It is well-known that a contraction can be written as a direct sum of a unitary operator and a completely non-unitary contraction.  The result is referred to as the \textit{canonical decomposition of a contraction} (see CH-I in \cite{NagyFoias}). In this Section, we present an analog of this result for operators in $C_{1,r}$ class. To do so, we first define an analog of completely non-unitary contraction for operators in $C_{1, r}$ class.

\begin{defn}
	An operator $T \in \Q$ acting on a Hilbert space $\HS$ is said to be \textit{c.n.u. $\Q$}-contraction if $\HS$ has no non-zero closed subspace that reduces $T$ to an operator in $\mathcal{C}_{1, r}$. An operator $T \in QA_r$ acting on a Hilbert space $\HS$ is said to be \textit{c.n.u. $QA_r$}-contraction if $\HS$ has no non-zero closed subspace that reduces $T$ to an operator in $\mathcal{Q}A_r$. 
\end{defn}

We have by Lemma \ref{lem_basic} and Proposition \ref{prop_QSA} that an operator $T \in C_{1,r}$ is a c.n.u. $C_{1, r}$-contraction if and only if $r^{-1\slash 2}T$ (which must belong to $QA_r)$ is a c.n.u. $QA_r$-contraction. For this reason, one can easily construct c.n.u. $QA_r$-contractions from the c.n.u. $C_{1, r}$-contractions.  We now present a few examples of c.n.u. $C_{1, r}$-contractions.

\begin{eg}
	Let $n \in \mathbb{N}$ and $r \in (0,1)$. Consider the operator $T_n=r^{1\slash 2n}I_\HS$ acting on a Hilbert space $\HS$. Evidently, $(T_1, \dotsc, T_d)$ is a doubly commuting (in fact normal) tuple of operators in $\Q$ for any $d \in \mathbb{N}$. Note that
	$
	(1+r^2)I_\HS-T_n^*T_n-r^2T_n^{-1}T_n^{-*}=1+r^2-r^{1\slash n}-r^2r^{-1\slash n}=(1-r^{1\slash n})(1-r^2r^{-1\slash n})>0.
	$
	Hence, $(T_1, \dotsc, T_d)$ is a doubly commuting tuple of c.n.u. $\Q$-contraction.\qed
\end{eg}

The idea of the next example is due to Sarason \cite{Sarason}.	

\begin{eg}
	Consider the space $L^2(\partial\mathbb{A}_r)=L^2(\mathbb{T})\oplus L^2(r\mathbb{T})$	where $L^2(\mathbb{T})$ and $L^2(r\mathbb{T})$ are endowed with normalized Lebesgue measure and the Hilbert norm is given by
	\begin{equation*}
		\|f\|^2=\frac{1}{2\pi}\overset{2\pi}{\underset{0}{\int}}|f(e^{it})|^2dt+\frac{1}{2\pi}\overset{2\pi}{\underset{0}{\int}}|f(re^{it})|^2dt  .  
	\end{equation*}
	For every $0 \leq \alpha < 1,$ define the function $w_\alpha$ on $\partial \mathbb{A}_r$ as
	\[
	w_\alpha(e^{it})=e^{i\alpha t} \quad \text{and} \quad  w_\alpha(re^{it})=r^\alpha e^{i\alpha t} \quad \text{for}\ \ 0 \leq t < 2\pi.
	\]
	Define the operator $S$ on $L^2(\partial\mathbb{A}_r)$ by $S(f)(z)=zf(z)$. Then $Sw_{\alpha}=w_{\alpha+1}$ and $S^{-1}w_\alpha=w_{\alpha-1}$. The set $\{w_{\alpha+n}: n \in \mathbb{Z}\}$ is an orthogonal set in $L^2(\partial \mathbb{A}_r)$. Let $H_{\alpha}^2(\partial \mathbb{A}_r)$ be the span closure of $\{w_{\alpha+n} : n \in \mathbb{Z}\}$ in $L^2(\partial \mathbb{A}_r)$. Define
	\[
	e_n=\frac{w_{\alpha+n}}{\sqrt{1+r^{2(\alpha+n)}}} \quad \text{and} \quad 	\alpha_n=\sqrt{\frac{1+r^{2(\alpha+n+1)}}{1+r^{2(\alpha+n)}}}. 
	\]
	The set $\{e_n : n \in \mathbb Z\}$ forms an orthonormal basis of $H_{\alpha}^2(\partial \mathbb{A}_r)$. It is not difficult to see that
	\[
	Se_n=\alpha_ne_{n+1}, \quad S^{-1}e_n=\frac{1}{\alpha_{n-1}}e_{n-1}, \quad S^*e_n=\alpha_{n-1}e_{n-1} \quad \text{and} \quad S^{-*}e_n=\frac{1}{\alpha_{n}}e_{n+1}.
	\]
	Evidently, $H_{\alpha}^2(\partial \mathbb{A}_r)$ is invariant under $S$ and $S^{-1}$. Let $S_\alpha=S|_{H_{\alpha}^2(\partial \mathbb{A}_r)}$. Then
	\begin{equation}\label{eqn_egS}
		1+r^2-|\alpha_n|^2-r^2|\alpha_n|^{-2}=\frac{r^{2(\alpha+n)}(1-r^2)^2}{(1+r^{2(\alpha+n)})(1+r^{2(\alpha+n+1)})}>0.
	\end{equation}
	Take any $\displaystyle f=\overset{\infty}{\underset{-\infty}{\sum}}f_ne_n$ in $H_\alpha^2(\partial \A)$. Then
	\begin{equation*}
		\begin{split}
			(1+r^2)\|f\|^2-\|S_\alpha f\|^2-r^2\|S_\alpha^{-*}f\|^2
			%&= %(1+r^2)\overset{\infty}{\underset{-\infty}{\sum}}|f_n|^2-\overset{\infty}{\underset{-\infty}{\sum}}|f_n|^2|\alpha_n|^2-r^2\overset{\infty}{\underset{-\infty}{\sum}}|f_n|^2|\alpha_n|^{-2}\\
			&= \overset{\infty}{\underset{-\infty}{\sum}}\left((1+r^2)-|\alpha_n|^2-r^2|\alpha_n|^{-2}\right)|f_n|^2,
		\end{split}
	\end{equation*}
which	is always non-negative. Also, we have by (\ref{eqn_egS}) that $(1+r^2)\|f\|^2-\|S_\alpha f\|^2-r^2\|S_\alpha^{-*}f\|^2=0$ if and only if $f=0$. Consequently, $S_\alpha$ is a c.n.u. $\Q$-contraction.\qed
\end{eg}

\begin{thm}\label{thm_canon}$($Canonical decomposition of operators in $\Q).$
	Let $T$ be an operator in $\Q$ acting on a Hilbert space $\HS$. Then there exists a unique orthogonal decomposition of $\HS=\HS_1 \oplus \HS_2$ into closed subspaces $\HS_1 \oplus \HS_2$ reducing $T$ such that the following holds: 
	\begin{enumerate}
		\item $T|_{\HS_1} \in \mathcal{C}_{1, r}$;
		\item $T|_{\HS_2}$ is a c.n.u. $\Q$-contraction.
	\end{enumerate}
	Moreover,  $\HS_1$ is the maximal closed $T$ reducing subspace restricted to which $T \in \mathcal{C}_{1, r}$ and the subspace $\HS_1$ is given by 
	\[
	\bigcap_{k \in \mathbb{N}}\ \bigcap_{n, m \in (\mathbb{N} \cup \{0\})^k} \left\{x \in \HS :  (1+r^2)\|p_{k, nm}(T)x\|^2-\|Tp_{k, nm}(T)x\|^2-r^2\|T^{-*}p_{k, nm}(T)x\|^2  \right\}
	\]
	where 
	\[
	p_{k, nm}(T)=T^{n_1}T^{*m_1}\dotsc T^{n_k}T^{*m_k} \quad \text{for} \quad n=(n_1, \dotsc, n_k), m=(m_1, \dotsc, m_k) \in (\mathbb{N}\cup \{0\})^k, k \in \mathbb{N}.
	\]
	%\[
	%\bigcap \ Ker \bigg[(1+r^2)p_{k, nm}(T)^*p_{k, nm}(T)-(Tp_{k,nm}(T))^*Tp_{k, nm}(T)-r^2(T^{-*}p_{k, nm}(T))^*T^{-*}p_{k, nm}(T) \bigg]
	%\]
	
\end{thm}

\begin{proof}
	For $x \in \HS, k \in \mathbb{N}$ and $n, m \in (\mathbb{N} \cup \{0\})^k$, we define 
	\[
	\delta_{k, nm}(T)(x)=(1+r^2)\|p_{k, nm}(T)x\|^2-\|Tp_{k, nm}(T)x\|^2-r^2\|T^{-*}p_{k, nm}(T)x\|^2
	\]
	and then consider
	\[
	\HS_1=\left\{x \in \HS : \delta_{k, nm}(T)(x)=0 \ \ \text{for} \ n, m \in (\mathbb{N} \cup \{0\})^k, k \in \mathbb{N} \right\}.
	\] 
	We show that $\HS_1$ is a reducing subspace for $T$. Let $x \in \HS_1, k \in \mathbb{N}$ and $n, m \in (\mathbb{N} \cup \{0\})^k$. Then 
	\[
	\delta_{k, nm}(T)(Tx)=\delta_{k+1, ij}(T)(x)=0 \qquad \text{and} \qquad \delta_{k, nm}(T)(T^*x)=\delta_{k+1, \alpha\beta}(T)(x)=0
	\]
	for $i=(n_1, \dotsc, n_k, 1), j=(m_1, \dotsc, m_k, 0), \alpha=(n_1, \dotsc, n_k, 0)$ and $\beta=(m_1, \dotsc, m_k, 1)$. Consequently, $Tx, T^*x \in \HS_1$ and so, $\HS_1$ is a closed reducing subspace for $T$. Let $x \in \HS_1$. For $k=1, n=0$ and $m=1$, we have that $p_{k, nm}(T)=T^*$. Then
	\[
	\delta_{k, nm}(T)(x)=(1+r^2)\|T^*x\|^2-\|TT^*x\|^2-r^2\|x\|^2=0
	\]
	and so, $T|_{\HS_1}$ is in $\mathcal{C}_{1, r}$. Let $\mathcal{L} \subseteq \HS$ be a closed reducing subspace for $T$ such that $T|_\mathcal{L}$ is in $\mathcal{C}_{1, r}$. Then $x_{k, nm}=p_{k, nm}(T)x \in \mathcal{L}$ for all $x \in \mathcal{L}, k \in \mathbb{N}$ and $n, m \in (\mathbb{N} \cup \{0\})^k$. Thus $0=(1+r^2)\|x_{k, nm}\|^2-\|Tx_{k, nm}\|^2-r^2\|T^{-*}x_{k, nm}\|^2=\delta_{k, nm}(T)(x)$ and so, $x \in \HS_1$. Hence, $\HS_1$ is the maximal closed $T$ reducing subspace restricted to which $T \in \mathcal{C}_{1, r}$. Define $\HS_2=\HS \ominus \HS_1$. It is evident that $\HS_2$ reduces $T$ and is a closed subspace of $\HS$. If $\HS_2$ has a closed reducing subspace $\mathcal{L}$ restricted to which $T \in \mathcal{C}_{1, r}$, then $\mathcal{L} \subseteq \HS_1 \cap \HS_2=\{0\}$. Hence, $T|_{\HS_2}$ is a c.n.u. $\Q$-contraction. The uniqueness of the decomposition follows again from the maximality of $\HS_1$. Indeed, if $\HS=\mathcal{L}_1 \oplus \mathcal{L}_2$ is another decomposition such that $(1)$ and $(2)$ of the statement hold, then by maximality of $\HS_1$, we have that $\mathcal{L}_1 \subseteq \mathcal{H}_1$. The space $\HS_1 \ominus \mathcal{L}_1$ reduces $T$ to an operator in $ \mathcal{C}_{1, r}$, because $\HS_1$ and $\mathcal{L}$ reduce $T$ to operators in $\mathcal{C}_{1, r}$. Also, $\HS_1 \ominus \mathcal{L}_1 \subseteq \HS \ominus \mathcal{L}_1=\mathcal{L}_2$. Since $T|_{\mathcal{L}_2}$ is a c.n.u. $\Q$-contraction, $\HS_1 \ominus \mathcal{L}_1=\{0\}$, i.e., $\HS_1= \mathcal{L}_1$ which completes the uniqueness part. 
\end{proof}

We now extend the decomposition result from Theorem \ref{thm_canon} to any tuple of doubly commuting operators in the $C_{1,r}$ class. The next lemma is essential for proving this result.

\begin{lem}\label{lem_reduce}
	For a pair of doubly commuting operators $A$ and $B$ in $\Q$ acting on a Hilbert space $\HS$, if $A=A_1 \oplus A_2$ is the canonical decomposition of $A$ $($as in Theorem \ref{thm_canon}$)$ with respect to the decomposition $\HS=\HS_1 \oplus \HS_2$, then $\HS_1, \HS_2$ are reducing subspaces for $B$.
\end{lem}

\begin{proof}
	Following the proof of Theorem \ref{thm_canon}, we have that
	\[
	\HS_1=\left\{x \in \HS : \delta_{k, nm}(A)(x)=0 \quad \text{for} \ n, m \in (\mathbb{N} \cup \{0\})^k, k \in \mathbb{N} \right\},
	\]
	where 
	\[
	\delta_{k, nm}(A)(x)=(1+r^2)\|p_{k, nm}(A)x\|^2-\|Ap_{k, nm}(A)x\|^2-r^2\|A^{-*}p_{k, nm}(A)x\|^2.
	\]
	For $k \in \mathbb{N}$ and $n, m \in (\mathbb{N} \cup \{0\})^k$, we define 
	\begin{equation*}
		\Delta_{k, nm}(A)=(1+r^2)p_{k, nm}(A)^*p_{k, nm}(A)-(Ap_{k, nm}(A))^*Ap_{k, nm}(A)-r^2(A^{-*}p_{k, nm}(A))^*A^{-*}p_{k, nm}(A).
	\end{equation*}
	Let $x \in \HS, k \in \mathbb{N}$ and $n, m \in (\mathbb{N} \cup \{0\})^k$. Since $A \in \Q$ and $p_{k, nm}(A)x \in \HS$, we have that 
	\begin{equation*}
		\delta_{k, nm}(A)(x)=(1+r^2)\|p_{k, nm}(A)x\|^2-\|Ap_{k, nm}(A)x\|^2-r^2\|A^{-*}p_{k, nm}(A)x\|^2 \geq 0.
	\end{equation*}
	It is easy to see that $\la \Delta_{k, nm}(A)x, x \ra =\delta_{k, nm}(A)x \geq 0$. Thus, $\Delta_{k, nm}(A)$ is a positive operator and hence, admits a unique positive square root denoted by $\Delta_{k, nm}(A)^{1\slash 2}$. We now define
	\[
	\HS_1'=\bigcap_{\substack{ k \in \mathbb{N} \\ n, m \in (\mathbb{N}\cup \{0\})^k}}  Ker \ \Delta_{k, nm}(A).
	\]
	We claim that $\HS_1=\HS_1'$. Let $x \in \HS_1', k \in \mathbb{N}$ and $n, m \in (\mathbb{N} \cup \{0\})^k$. Note that 
	$
	\delta_{k, nm}(A)x=\la \Delta_{k, nm}(A)x, x \ra=\|\Delta_{k, nm}(A)^{1\slash 2}x\|^2
	$
	and so, $\delta_{k, nm}(A)x=0$ if and only if $\Delta_{k, nm}(A)x=0$. Therefore, $\HS_1=\HS_1'$. Since $A$ and $B$ doubly commute, each $\Delta_{k, nm}(A)$ commute with $B$ and $B^*$. Then
	\[
	\Delta_{k,nm}(A)Bx=B\Delta_{k,nm}(A)x=0 \quad \text{and} \quad \Delta_{k,nm}(A)B^*x=B^*\Delta_{k,nm}(A)x=0
	\]
	for all $k \in \mathbb{N}, n, m \in (\mathbb{N} \cup \{0\})^k$ and $x \in \HS_1$. Thus $Bx, B^*x \in \HS_1$ and so, $\HS_1$ reduces $B$. Consequently, $\HS_2=\HS \ominus \HS_1$ also reduces $B$  which completes the proof.
\end{proof}

We now present a decomposition result for a tuple of doubly commuting operators in $\Q$ class. To do so, we shall adopt the following terminologies for the rest of the article.

\begin{defn}
	An operator $T \in \Q$ acting on a Hilbert space $\HS$ is said to be of \textit{type $t_1$} if $T$ is in $\mathcal{C}_{1, r}$ and of \textit{type $t_2$} if $T$ is a c.n.u. $\Q$-contraction. 
\end{defn}

\begin{thm}\label{thm_dc_C1r}
	Let $(T_1, \dotsc, T_d)$ be a doubly commuting tuple of operators in $\Q$ acting on a Hilbert space $\HS$. Then there exists a decomposition of $\HS$ into an orthogonal sum of $2^d$ subspaces $\HS_1, \dotsc, \HS_{2^d}$ of $\HS$ such that
	\begin{enumerate}
		\item every $\HS_j
		, 1 \leq j \leq 2^d$, is a joint reducing subspace for $T_1, \dotsc, T_d$ ;
		\item $T_i|_{\HS_j}$ is either in $\mathcal{C}_{1, r}$ or a c.n.u. $\Q$-contraction for $ 1 \leq i \leq d$ and $1 \leq j \leq 2^d$;
		\item if $\Omega_d$ is the collection of all functions $\omega:\{1, \dotsc, d\} \to \{t_1, t_2\}$, then corresponding to every $\omega \in \Omega_d$, there is a unique subspace $\HS_\ell (1 \leq \ell \leq 2^d)$ such that the $j$-th component of $(T_1|_{\HS_{\ell}}, \dotsc, T_d|_{\HS_{\ell}})$ is of the type $\omega(j)$ for $1 \leq j \leq d$;
		\item $\HS_1$ is the maximal joint reducing subspace for $T_1, \dotsc, T_d$ such that $T_1|_{\HS_1}, \dotsc, T_d|_{\HS_1}$ are in $\mathcal{C}_{1, r}$;
		\item $\HS_{2^d}$ is the maximal joint reducing subspace for $T_1, \dotsc, T_d$ such that $T_1|_{\HS_{2^d}}, \dotsc, T_d|_{\HS_{2^d}}$ are c.n.u. $\Q$-contractions.
	\end{enumerate}
	The decomposition is uniquely determined. One or more members of $\{\HS_j : 1 \leq j \leq 2^d\}$ may equal the trivial subspace $\{0\}$.
\end{thm}

\begin{proof}
	To make the algorithm clear, we prove the decomposition result for $d=2$. Let $(T_1, T_2)$ be a doubly commuting pair of operators in $\Q$ acting on a Hilbert space $\HS$. It follows from Theorem \ref{thm_canon} that there exist orthogonal closed reducing subspaces $\HS_1$ and $\HS_2$ for $T_1$ in $\HS$ such that $\HS=\HS_1 \oplus \HS_2$. Moreover, $T_1|_{\HS_1}$ is in $\mathcal{C}_{1, r}$ and $T_1|_{\HS_2}$ is a c.n.u. $\Q$-contraction. It follows from Lemma \ref{lem_reduce} that $\HS_1$ and $\HS_2$ also reduce $T_2$. Since $T_2|_{\HS_1}$ is also in $\Q$, it follows from Theorem \ref{thm_canon} that $\HS_1=\HS_{11} \oplus \HS_{12}$ where $\HS_{11}, \HS_{12}$ are orthogonal reducing subspaces for $T_2$. Also, $T_2|_{\HS_{11}}$ is in $\mathcal{C}_{1, r}$ and $T_2|_{\HS_{12}}$ is a c.n.u. $\Q$-contraction. Similarly, we obtain that $\HS_2=\HS_{21} \oplus \HS_{22}$ and $\HS_{21}, \HS_{22}$ reduce $T_2$. Moreover, $T_2|_{\HS_{21}} \in \mathcal{C}_{1, r}$ and  $T_2|_{\HS_{22}}$ is a c.n.u. $\Q$-contraction. We have by Lemma \ref{lem_reduce} that all the subspaces $\HS_{11}, \HS_{12}, \HS_{21}, \HS_{22}$ reduce both $T_1$ and $T_2$. Hence, $\HS=\HS_{11} \oplus \HS_{12} \oplus \HS_{21} \oplus \HS_{22}$ is an orthogonal decomposition into closed reducing subspaces for $T_1$ and $T_2$ such that the following holds:
	
	\begin{enumerate}
		\item $T_1|_{\HS_{11}}, T_2|_{\HS_{11}}$ are in $\mathcal{C}_{1, r}$;
		\item $T_1|_{\HS_{12}} \in \mathcal{C}_{1, r}$ and $T_2|_{\HS_{12}}$ is a c.n.u. $\Q$-contraction;
		\item $T_1|_{\HS_{21}}$ is a c.n.u. $\Q$-contraction and $T_2|_{\HS_{21}} \in \mathcal{C}_{1, r}$;
		\item $T_1|_{\HS_{22}}, T_2|_{\HS_{22}}$ are c.n.u. $\Q$-contractions.
	\end{enumerate}
	
	Thus, the desired conclusion follows for $d=2$. The general case follows from the principle of mathematical induction and Lemma \ref{lem_reduce}. Indeed, let the desired conclusion hold for $d=k$ with $k \geq 2$. We show that the desired decomposition result holds for any doubly commuting $(k+1)$-tuple $(T_1, \dotsc, T_k, T_{k+1})$ of operators in $\Q$ acting on a Hilbert space $\HS$. Since $T_{k+1} \in \Q$ and it doubly commutes with $T_1, \dotsc, T_k$,  a repeated application of Lemma \ref{lem_reduce} yields that  each subspace $\mathcal{H}_{j}$ reduces $T_{k+1}$ for $1 \leq j \leq 2^{k}$. Consider the collection of operator tuples 
	\[
	\Lambda=		\left\{(T_1|_{\mathcal{H}_j}, \dotsc, T_k|_{\mathcal{H}_j}, \; T_{k+1}|_{\mathcal{H}_j}) \ : \  1 \leq j \leq 2^k \right\}
	\]
	and let
	$
	(T_1|_{\mathcal{H}_j}, \dotsc, T_k|_{\mathcal{H}_j}, \; T_{k+1}|_{\mathcal{H}_j})
	$ 
	be an arbitrary tuple from this collection. By induction hypothesis, each $T_i|_{\mathcal{H}_j} (1 \leq i \leq k)$ is either in $\mathcal{C}_{1, r}$ or a c.n.u. $\Q$-contraction. We apply Theorem \ref{thm_canon} on  each $T_{k+1}|_{\mathcal{H}_j}$ and so, $\mathcal{H}_j$ can be written as an orthogonal sum of two reducing subspaces of $T_{k+1}$, say $\mathcal{H}_{j1}, \mathcal{H}_{j2}$, such that $T_{k+1}|_{\mathcal{H}_{j1}} \in \mathcal{C}_{1, r}$ and $T_{k+1}|_{\mathcal{H}_{j2}}$ is a c.n.u. $\Q$-contraction. Again, by Lemma \ref{lem_reduce}, $\mathcal{H}_{j1}, \mathcal{H}_{j2}$ reduce each of $T_1|_{\mathcal{H}_{j}}, \dotsc, T_k|_{\mathcal{H}_{j}}$. Hence, each tuple 
	$
	(T_1|_{\mathcal{H}_j}, \dotsc, T_k|_{\mathcal{H}_j},  T_{k+1}|_{\mathcal{H}_j})
	$
	orthogonally decomposes into two parts. Continuing this way, for every tuple in the collection $\Lambda$, we get $2^{k+1}$ orthogonal subspaces and the rest follows from the induction hypothesis.  
\end{proof}

\begin{rem} We have by Lemma \ref{lem_basic} that an operator $T \in QA_r$ if and only if $rT \in C_{1, r^2}$ and so, an analogue of the decomposition result in Theorem \ref{thm_dc_C1r} holds for doubly commuting tuple of operators in $QA_r$ as well. To do so, one needs to replace the class $\mathcal{C}_{1,r}$ by $\mathcal{Q}A_r$ and c.n.u. $C_{1, r}$-contractions by c.n.u. $QA_r$-contractions in the statement of Theorem \ref{thm_dc_C1r}. 
\end{rem} 

\vspace{0.1cm}

An operator $T$ is said to be an \textit{$\A$-contraction} if $\overline{\mathbb{A}}_r$ is a spectral set for $T$. Evidently, an $\A$-contraction is in $\Q$. In fact, any operator of the form $T=UB \in C_{1,r}$, where $U$ is a unitary and $B$ is an $\A$-contraction. It was proved in \cite{N-S1} that any operator in $\Q$ class is of this form. We extend this result to a doubly commuting tuple of operators in $\Q$ class providing a characterization of finitely many doubly commuting operators in $C_{1, r}$ class in terms of $\A$-contractions.

\begin{thm}\label{thm_dc_decomp}
	Let $\underline{T}=(T_1, \dotsc, T_d)$ be a tuple of  invertible operators acting on a Hilbert space $\HS$. Then $\underline{T}$ is a doubly commuting tuple of operators in $\Q$  if and only if there exist a tuple $\underline{U}=(U_1, \dotsc, U_d)$ of commuting unitaries and a tuple $\underline{D}=(D_1, \dotsc, D_d)$ of commuting self-adjoint $\A$-contractions on $\HS$ such that $D_iU_j=U_jD_i$ and $T_j=U_jD_j$ for $1 \leq i,j \leq d$ with $i \ne j$.
\end{thm}

\begin{proof}
	Let $\underline{T}$ be a doubly commuting tuple of operators in $\Q$. We have the polar decomposition 
	$
	T_j=U_j(T_j^*T_j)^{1\slash 2}$ for $j=1, \dotsc, d
	$,
	where each $U_j=T_j(T_j^*T_j)^{-1\slash 2}$ is a unitary. Let $D_j=(T_j^*T_j)^{1\slash 2}$ and choose $i, j \in \{1, \dotsc, d \}$ with $i \ne j$. Using the doubly commutativity hypothesis, it is easy to see that $
		D_i^2D_j^2=D_j^2D_i^2$ and $T_iD_j^2=D_j^2T_i.
$ We have by spectral theorem that $D_iD_j=D_jD_i$ and $T_iD_j=D_jT_i$. Hence, $U_iU_j=T_jD_j^{-1}T_iD_i^{-1}=T_iD_i^{-1}T_jD_j^{-1}=U_jU_i$ and $U_iD_j=T_iD_i^{-1}D_j=D_jT_iD_i^{-1}$. Putting everything together, we have that $\underline{U}$ and $\underline{D}$ consist of commuting operators and $U_iD_j=D_jU_i$ for $1 \leq i, j \leq d$ with $i \ne j$. Since each $D_j$ is a self-adjoint operator in $\Q$, it follows that $D_j$ is an $\A$-contraction, because $\sigma(D_j) \subseteq \overline{\mathbb{A}}_r$ is a spectral set for each $D_j$. The latter follows from the fact that any compact set containing $\sigma(N)$ is a spectral set for a normal operator $N$.
	
	\medskip 
	
	Conversely, assume that there exist a tuple $\underline{U}=(U_1, \dotsc, U_d)$ of commuting unitaries and a tuple $\underline{D}=(D_1, \dotsc, D_d)$ of commuting self-adjoint $\A$-contractions on $\HS$ such that $D_iU_j=U_jD_i$ and $T_j=U_jD_j$ for $1 \leq i, j \leq d$ with $i \ne j$. Then $\|T_j\| \leq \|D_j\| \leq 1$ and $\|rT_j^{-1}\| \leq \|rD_j^{-1}\| \leq 1$ for $1 \leq j \leq d$. Moreover, we have
	$
	T_iT_j=U_iD_iU_jD_j=U_jD_jU_iD_i=T_jT_i$ and $T_iT_j^*=U_iD_iD_jU_j^*=D_jU_iD_iU_j^*=D_jU_j^*U_iD_i=T_j^*T_i$
	for all $1 \leq i, j \leq d$ with $i \ne j$. The proof is now complete.
\end{proof}

Needless to mention, we have the following analog of Theorem \ref{thm_dc_decomp} for $QA_r$ class. The proof follows directly from Lemma \ref{lem_basic} and Theorem \ref{thm_dc_decomp}. We skip the proof.

\begin{thm}\label{thm_dc_decompII}
	Let $\underline{T}=(T_1, \dotsc, T_d)$ be a tuple of  invertible operators acting on a Hilbert space $\HS$. Then $\underline{T}$ is a doubly commuting tuple of operators in $QA_r$  if and only if there exist a tuple $\underline{U}=(U_1, \dotsc, U_d)$ of commuting unitaries and a tuple $\underline{D}=(D_1, \dotsc, D_d)$ of commuting self-adjoint operators on $\HS$ such that $\sigma(D_j) \subseteq \overline{A}_r, \ D_iU_j=U_jD_i$ and $T_j=U_jD_j$ for $1 \leq i,j \leq d$ with $i \ne j$.
\end{thm}


\begin{thebibliography}{9}
		
		\vspace{0.4cm}
		
		
		\bibitem{Agler}
		J. Agler, \textit{Rational dilation on an annulus,} Ann. of Math., 121 (1985),  537 -- 563. \medskip
		
		
		\bibitem{AglerII}
		J. Agler, \textit{The Arveson extension theorem and coanalytic models}, Integral Equations and Operator Theory, 5 (1982), 608 -- 631. \medskip		
		
		
		\bibitem{Arveson}
		W. B. Arveson, \textit{Subalgebras of C$^*$-algebras}, Acta Math., 123 (1969),  141 -- 224. \medskip
		
		\bibitem{Ball}
		J. A. Ball, \textit{A lifting theorem for operators models of finite rank on multiply connected domains}, J. Operator theory, 1 (1979), 3 -- 25. \medskip
		
		
			\bibitem{Dmitry}
		G. Bello and D. Yakubovich, \textit{An operator model in the annulus}, J. Operator Theory, 90 (2023), 25 -- 40. \medskip
		
		
		\bibitem{Conway_OT}
		J. B. Conway, \textit{A Course in Operator Theory}, American Mathematical Society (2000). \medskip 
		
	
				
			\bibitem{Douglas}
		R. G. Douglas and V. I. Paulsen, \textit{Completely bounded maps and hypo-Dirichlet algebras},  Acta Sci. Math. (Szeged), 50 (1986), 143 -- 157. \medskip 		
				
		\bibitem{Nagy}
		C. Foias and B. Sz.-Nagy, \textit{Sur les contractions de l’espace de Hilbert IV}, Acta Sci. Math., 21 (1960), 251 -- 259. \medskip		
		
			\bibitem{Fuglede}
		B. Fuglede, \textit{A commutativity theorem for normal operators}, Proc. Nat. Acad. Sci., 36 (1950), 36 -- 40. \medskip 
		
		\bibitem{Furuta}
		T. Furuta, \textit{On the polar decomposition of an operator}, Acta Sci. Math. (Szeged), 46 (1983), 261 -- 268. \medskip	
		
		\bibitem{McColloughIII}	 
		B. A. Mair and S. McCullough, \textit{Invariance of extreme harmonic functions on an annulus; applications to theta functions}, Houston J. Math., 20 (1994), 453 -- 473. \medskip
		
		\bibitem{McColloughII}	
		S. McCullough, \textit{Matrix functions of positive real part on an annulus}, Houston J. Math., 21 (1995), 489–506. \medskip
		
		\bibitem{Pas-McCull}
		S. McCullough and J. E. Pascoe, \textit{Geometric Dilations and Operator Annuli}, J. Funct. Anal., 285 (2023), Paper No. 110035, 20 pp. \medskip		
		
			\bibitem{Misra}
		G. Misra, \textit{Curvature inequalities and extremal properties of bundle shifts}, J. Operator Theory, 11 (1984), 305 -- 317. \medskip
		
		\bibitem{Mittal}
		M. Mittal, \textit{Function theory on the quantum annulus and other domains}, Thesis (Ph.D.)-University of Houston, ISBN: 978-1124-46385-8, 2010, 141 pp. \medskip		
		
		\bibitem{Murphy}
		G. J. Murphy, \textit{$C^*$-algebras and operator theory}, Academic press, 2014. \medskip
		
		\bibitem{NagyFoias}
		B. Sz.-Nagy, C. Foias, L. Kerchy and H. Bercovici, Harmonic analysis of operators on Hilbert space, \textit{Universitext Springer}, New York, 2010. \medskip		
		
		\bibitem{B-Nagy} Bela Sz.-Nagy, \textit{Sur les contractions de l’espace de Hilbert}, Acta Sci. Math., 15 (1953), 87 -- 92. \medskip
		
		\bibitem{Nelson}
		E. Nelson, \textit{The distinguished boundary of the unit operator ball}, Proc. Amer. Math. Soc., 12 (1961), 994 -- 995. \medskip
		
				\bibitem{N-S1}
		S. Pal and N. Tomar, \textit{Characterizations and models for the $C_{1, r}$ class and quantum annulus}, Canad. Math. Bull. 68 (2025), 818--833. \medskip
		
		
		\bibitem{Paulsen}
		V. Paulsen, \textit{Completely bounded maps and operator algebras}, Cambridge University Press, 2003. \medskip
		
		\bibitem{Sarason}
		D. Sarason, \textit{The $H^p$ spaces of an annulus}, Mem. Amer. Math. Soc., 56 (1965), 78 pp. \medskip
		
			\bibitem{Shields}
		A. L. Shields, \textit{Weighted shift operators and analytic function theory},
		Topics in operator theory, Math. Surveys, No. 13. Amer. Math. Soc.,
		Providence, RI, 1974, 49 -- 128. \medskip 
		
		
		\bibitem{Stine}
		W. F. Stinespring, \textit{Positive functions on $C^*$-algebras}, Proc. Amer. Math. Soc., 6 (1955), 211 -- 216. \medskip			
		
		\bibitem{Taylor}
		J. L. Taylor, \textit{The analytic-functional calculus for several commuting operators}, Acta Math., 125 (1970), 1 -- 38. \medskip
		
		
	\bibitem{TsikalasII}
	G. Tsikalas, \textit{A note on a spectral constant associated with an annulus},
	Oper. Matrices, 16 (2022), 95 -- 99.      \medskip
		
		
	\end{thebibliography}
\end{document}